\documentclass[12pt,reqno]{amsart}

\usepackage{amsmath,amssymb,amscd,bm, graphicx, amsthm,
  enumerate, hyperref}
  \usepackage{framed}
\calclayout

\usepackage{mathrsfs}
\usepackage{mathtools}
\usepackage[all]{xy}
\usepackage{amsmath}
\usepackage{color,xcolor}
\definecolor{darkred}{rgb}{1,0,0} 
\definecolor{darkgreen}{rgb}{0,0.8,0}
\definecolor{darkblue}{rgb}{0,0,1}
\usepackage{mathtools}

\DeclarePairedDelimiter\floor{\lfloor}{\rfloor}

\hypersetup{colorlinks,
linkcolor=darkblue,
filecolor=darkgreen,
urlcolor=darkred,
citecolor=darkred}

\numberwithin{equation}{section}

\theoremstyle{plain}

\theoremstyle{plain}
\newtheorem{theorem}{Theorem}
\numberwithin{theorem}{section}
\newtheorem{proposition}[theorem]{Proposition}
\newtheorem{lemma}[theorem]{Lemma}

\theoremstyle{definition}

\theoremstyle{definition}
\newtheorem{remark}[theorem]{Remark}
\newtheorem{definition}[theorem]{Definition}
\newtheorem{notation}[theorem]{Notation}

\newcommand{\interior}[1]{%
  {\kern0pt#1}^{\mathrm{o}}%
}

      \usepackage{mdframed}

\newtheoremstyle{named}{}{}{\itshape}{}{\bfseries}{.}{.5em}{\thmnote{#3 }#1}
\theoremstyle{named}

\DeclareMathOperator{\moda}{mod}

\newcommand{\norm}[1]{\left\lVert#1\right\rVert}

\DeclareMathOperator{\vol}{vol}

\title[Spacings between the fractional parts of $\MakeLowercase{n}^{\MakeLowercase{d}}\alpha$]{The distribution of spacings between the fractional parts of $\boldsymbol{n^d\alpha}$}
\author{Martino Fassina}
\address{M.F.: Department of Mathematics, University of Illinois at Urbana-Champaign, 1409 West Green
Street, Urbana, IL 61801, USA}
\email{fassina2@illinois.edu}

\author{Sun Kim}
\address{S.K.: AORC, SungKyunKwan University, 2066 Seobu-ro, Suwon 16419, Korea}
\email{ssunny8079@hanmail.net}

\author{Alexandru Zaharescu}
\address{A.Z.: Department of Mathematics, University of Illinois at
Urbana-Cham\-paign  1409 West Green Street, Urbana, IL 61801, USA and
Simon Stoilow Institute of Mathematics of the Romanian Academy, P.O. Box 1-764, RO-014700 Bucharest, Romania}
\email{zaharesc@illinois.edu}

\begin{document}
\begin{abstract} We study the distribution of spacings between the fractional parts of $n^d\alpha$. For $\alpha$
of high enough Diophantine type 
we prove a necessary and sufficient condition for $n^d\alpha\mod 1, 1\leq n\leq N,$ to be Poissonian as $N\to \infty$ along a suitable subsequence.
\end{abstract}
\subjclass[2010]{Primary 11K06. Secondary 11J71.}
\keywords{Correlations, fractional parts, local spacings, Diophantine type, curves over finite fields.}
\maketitle
\section{Introduction}

Let $f(x)$ be a polynomial, and consider the sequence of fractional parts $(\{f(n)\})_{n\in\mathbb{N}}$. It is of considerable interest to study the distribution of the spacings between members of the sequence. This problem arose in the context of the distribution of spacings between the energy levels of integrable systems \cite{BT77, CGI87}. When $f(x)=\alpha x$, the spacings are essentially those of the energy levels of a two-dimensional harmonic oscillator (see \cite{PBG89}). In this case, the sequence is not random: for any $\alpha$ and $N$, the consecutive spacings of $n\alpha\mod 1, 1\leq n\leq N$, take at most three values (see \cite{S58} and \cite{Sw59}).

In the more challenging case $f(x)=\alpha x^d, d\geq 2$, Rudnick and Sarnak \cite{RS98} investigated the pair correlation function, which measures the density of differences between pairs of elements of the sequence. They proved that for almost all $\alpha$ the pair correlation function is Poissonian.
For another approach to this result see \cite{BZ00}.

For the case $d=2$ significantly more is known. Rudnick, Sarnak and one of the authors \cite{RSZ01}, \cite{Z03}
investigated higher order correlations of $n^2\alpha\mod 1$ (more details about these works will be discussed later in the paper).  Recently, the size of clusters of $n^2\alpha\mod 1$ played a relevant role in the work of Dunn and
one of the authors \cite{DZ19} on a second moment of central values of certain half integral weight Dirichlet series. See also the survey \cite{Sh12} for connections between this sequence and other related topics.

In the present paper we return to the original sequence $n^d\alpha \mod 1$ studied by Rudnick and Sarnak. We let $\alpha$ be an irrational number, $d$ an integer, $d\geq 2$, and consider the problem of
studying the distribution of local spacings between the elements of the sequence
$n^d\alpha\mod 1,$ with $1\leq n\leq N$.
For an integer $m\geq 2$, and a smooth compactly supported function $f\in C^{\infty}_c({\bf R}^{m-1})$, we
consider the $m$-level correlation sums
\begin{equation*}
R^{(m)}(N,d,\alpha,f)=\frac{1}{N}\sum_{\substack{1\leq n_1,\ldots,n_m\leq N\\n_j\text{ distinct}}}F_N(n_1^d\alpha-n_2^d\alpha,\dots,n^d_{m-1}\alpha-n_m^d\alpha),
\end{equation*}
where $F_N({\bf y})=\sum_{{\bf l}\in{\bf Z}^{m-1}}f(N({\bf l}+{\bf y}))$.
We say that the $m$-level correlation of the sequence $n^d\alpha\mod 1$ is Poissonian if for every $f\in C^{\infty}_c({\bf R}^{m-1})$ we have
\begin{equation*}
\lim_{N\to\infty} R^{(m)}(N,d,\alpha,f)=\int_{{\bf R}^{m-1}} f({\bf x})\, d{\bf x}.
\end{equation*}
We say that the $m$-level correlation of the sequence $n^d\alpha\mod 1$ is Poissonian along a sequence $N_j\rightarrow\infty$ if for every $f\in C^{\infty}_c({\bf R}^{m-1})$ we have
\begin{equation*}
\lim_{j\to\infty} R^{(m)}(N_j,d,\alpha,f)=\int_{{\bf R}^{m-1}} f({\bf x})\, d{\bf x}.
\end{equation*}
If the $m$-level correlation of $n^d\alpha\mod 1$ is Poissonian for any $m\geq 2$ along the same sequence $N_j\rightarrow \infty$, we simply say that $n^d\alpha\mod 1$ is {\bf Poissonian along $N_j$}.

Given $\alpha$ and a sequence of rationals $b_j/q_j\to\alpha$, we say that $n^d\alpha\mod 1$ is {\bf Poissonian with respect to} $(b_j/q_j)_{j\in\mathbb{N}}$ if there exists a sequence $\{N_j\}_{j\to\infty}$ with $\frac{\log N_j}{\log q_j}\to 1$ such that $n^d\alpha\mod 1$ is  Poissonian along $N_j$.



We say that $\alpha$ is {\em not of finite Diophantine type} if there exists a sequence of triples $(b_j,q_j,k_j)$ of integers with $k_j\to\infty$ such that for every $j$ we have
\begin{equation}\label{sw}
\bigg{\vert}\alpha-\frac{b_j}{q_j}\bigg{\vert}\leq \frac{1}{q_j^{k_j}}.
\end{equation}


Our main objective is to prove the following surprising result. 
\begin{theorem}\label{curious}
Let $\alpha$ be an irrational number, not of finite Diophantine type, and let $(b_j,q_j)$ be a sequence such that $b_j/q_j\to\alpha$ as in \eqref{sw}. Then there are two alternatives:
\begin{enumerate}
\item Either $n^d\alpha\mod 1$ is Poissonian with respect to $(b_j/q_j)_{j\in\mathbb{N}}$ for {\bf every} $d\geq 2$;
\item or $n^d\alpha\mod 1$ is Poissonian with respect to $(b_j/q_j)_{j\in\mathbb{N}}$ for {\bf no} $d\geq 2$.
\end{enumerate}
\end{theorem}

Here the hypothesis that $\alpha$ is not of finite Diophantine type is used only in passing from the distribution of $n^d\alpha \mod 1$ to the distribution of $n^d(b_j/q_j)\mod 1$. We point out that for different values of $d$, the corresponding sequences $n^d(b_j/q_j)\mod 1$ have no obvious relations. Nevertheless, as we shall see later,
the {\bf same exact obstruction} to being Poissonian along a sequence applies
simultaneously for all $d\geq 2$. 


\section{A curve over a finite field}\label{SectionCurve}

Studying the spacing distribution of sequences of the form $n^db_j\,(\text{mod }q_j), 1\leq n\leq N$,
leads naturally to a point count on curves over finite fields. In this section we begin the investigation of such curves.

Let $q$ be a prime number, and let ${\bm Z}_q={\bm Z}/q{\bm Z}$ denote the field with $q$ elements. We let $k=\overline{{\bm Z}}_q$ be the algebraic closure of ${\bm Z}_q$. Let $m,d$ be integers, with $m\geq 2, d\geq 2.$ 
We will consider polynomials of degree $d$ in the ring $k[x_1,\dots, x_{m}]$. Let ${\bm a}=(a_1,\dots,a_{m-1})\in k^{m-1}$. For $j=1,\dots, m-1$, we define $g_j\in k[x_1,\dots,x_m]$ by
\begin{equation}\label{g}
g_j:=x_j^d-x_{j+1}^d-a_j.
\end{equation}
Let $\mathcal{C}(d,{\bm a},q)$ be the curve defined in $k^m$ by the system
\begin{equation}\label{system}
g_j=0,\quad j=1,\dots, m-1.
\end{equation}

The goal of this section is to prove the following criterion for the irreducibility of $\mathcal{C}(d,{\bm a},q)$.
\begin{proposition}\label{irred}
Assume that $q\nmid d$. Then the curve $\mathcal{C}(d,{\bm a},q)$ is irreducible in $k^{m}$ if and only if, for $i=1,\dots, m-1$, the partial sums $A_i=\sum^{m-1}_{k= i}a_k$ are all distinct and non-zero.
\end{proposition}

We recall some notions from commutative algebra. The {\em lexicographic order} is an order $>$ on the monomials of $k[x_1,\dots,x_m]$ such that $x_1^{\alpha_1}\cdots x_m^{\alpha_m}>x_1^{\beta_1}\cdots x_m^{\beta_m}$ exactly when the first non-zero entry of the vector $(\alpha_1-\beta_1,\dots,\alpha_m-\beta_m)$ is positive. For a polynomial $f\in k[x_1,\dots,x_m]$, we call its maximal monomial with respect to the lexicographic order the {\em initial term}. We denote the initial term of $f$ by $in(f)$. For a subset $S$ of a polynomial ring $k[x_1,\dots,x_m],$ we denote by $\langle S\rangle $ the ideal generated by $S$ in $k[x_1,\dots, x_m]$.
\begin{definition}
Let $I$ be an ideal in $k[x_1,\dots,x_m]$. A subset $G=\{g_1,\dots,g_s\}$ of $I$ is called a {\em Gr\"obner basis} of $I$ (with respect to the lexicographic order) if \[\big{\langle} in(f) \,\vert\, f\in I\big{\rangle}=\big{\langle} in(g_1),\dots, in(g_s)\big{\rangle}.\]
\end{definition}

\begin{lemma}\label{GB}
The set $G=\{g_1,\dots, g_{m-1}\}$, where the $g_j$ are defined as in \eqref{g}, is a Gr\"obner basis for the ideal $\langle g_1,\dots, g_{m-1} \rangle$.
\end{lemma}
\begin{proof}
For $1\leq i<j\leq m-1$, let $H_{ij}$ denote the \emph{s}-polynomial of the pair $(g_i,g_j)$ (see \cite[Definition 2.6]{St98}). That is, $H_{ij}$ is the unique linear combination of $g_i$ and $g_j$ canceling the initial terms $in(g_i), in(g_j)$, and whose coefficients are relatively prime monic monomials in $k[x_1,\dots,x_m]$. Hence,
\[H_{ij}=x_j^d(x_i^d-x_{i+1}^d-a_i)-x_i^d(x_j^d-x_{j+1}^d-a_j)=\bm{x_i^dx_{j+1}^d}+a_jx_i^d-x_{i+1}^dx_j^d-a_ix_j^d.\]
Here and in the following computations, the monomial in boldface is the initial term.
We now compute the remainder $R_G(H_{ij})$ of $H_{ij}$ by $G$ \cite[Definition 2.2]{St98}. Let $H^0=H_{ij}$. For $k\geq 1$, the polynomial $H^k$ is obtained by subtracting appropriate multiples of the elements of $G$ from $H^{k-1}$ in order to cancel its initial term $in(H^{k-1})$. We thus get
\[H^1=H^0-x_{j+1}^d(x_i^d-x_{i+1}^d-a_i)=\bm{a_j x_i^d}-x^d_{i+1}x_j^d+x^d_{i+1}x_{j+1}^d-a_ix_j^d+a_ix_{j+1}^d,\]
\[H^2=H^1-a_j(x_i^d-x_{i+1}^d-a_i)=\bm{-x_{i+1}^dx_j^d}+x_{i+1}^dx_{j+1}^d+a_jx_{i+1}^d-a_ix_j^d+a_ix^{d}_{j+1}+a_ia_j,\]
\[H^3=H^2+x_{i+1}^d(x_j^d-x_{j+1}^d-a_j)=\bm{-a_ix_j^d}+a_ix_{j+1}^d+a_ia_j,\]
\[H^4=H^3+a_i(x_j^d-x_{j+1}^d-a_j)=0.\]
Hence $R_G(H_{ij})=0$ for every $1\leq i<j\leq m-1$. The set $G$ is therefore a Gr\"obner basis by \cite[Proposition 2.7]{St98}.
\end{proof}

Let $R$ be an integral domain, and $I$ an ideal in the polynomial ring $R[x_1,\dots,x_m]$. P. Gianni, B. Trager, and G. Zacharias \cite{GTZ88} gave the following algorithm to check if $I$ is a prime ideal in $R[x_1,\dots,x_m]$. (See also \cite[Section 4 in Chapter 4]{AL94}).

\begin{framed}

\noindent
{\bf ALGORITHM}: Primality Test  \cite[Algorithm 4.4.1]{AL94}\\
{\bf Input:} An ideal $I$ in $R[x_1,\dots,x_m].$\\
{\bf Output:} TRUE if $I$ is a prime ideal, FALSE otherwise.\\
\noindent
Set $R_{m+1}=R$, and $R_i:=R[x_i,\dots,x_m]$ for $i=1,\dots,m$. \\Compute $J_i=I\cap R_i$ for $i=1,\dots, m+1$.\\
{\bf If} $J_{m+1}$ is not a prime ideal of $R$, {\bf then} result:= FALSE.\\
{\bf Else} result:= TRUE, $i:=m+1$.\\
\phantom{{\bf else} result:= TRUE,} {\bf While} $i>1$ {\bf and} result=TRUE {\bf do}\\
\phantom{{\bf else} result:= TRUE, {\bf Wh}} $R':=R_i/J_i$,\\
\phantom{{\bf else} result:= TRUE, {\bf Wh}} $J':=\text{ image of }J_{i-1}$ in $R'[x_{i-1}]$,\\
\phantom{{\bf else} result:= TRUE, {\bf Wh}} $k':=\text{ quotient field of }R'$.\\
\phantom{{\bf else} result:= TRUE, {\bf Wh}} Compute the polynomial $f$ such that $J'k'[x_{i-1}]=\langle f\rangle$.\\
\phantom{{\bf else} result:= TRUE, {\bf Wh}} {\bf If} $f$ is not zero and reducible over $k'$, {\bf then} result:= FALSE.\\
\phantom{{\bf else} result:= TRUE, {\bf Wh}} {\bf Else} compute $J'k'[x_{i-1}]\cap R'[x_{i-1}]$. \\
\phantom{{\bf else} result:= TRUE, {\bf Wh} {\bf Else} com} {\bf If} $J'k'[x_{i-1}]\cap R'[x_{i-1}]\neq J'$, {\bf then} result:=FALSE.\\
\phantom{{\bf else} result:= TRUE, {\bf Wh} {\bf Else} com} {\bf Else} $i:=i-1$.\\
{\bf Return} result.
\end{framed}

\bigskip

We will apply the algorithm to prove that, under the appropriate assumptions on ${\bm a}$, the ideal $I=\langle g_1,\dots,g_{m-1}\rangle$ is prime in $k[x_1,\dots,x_m]$. In our case, $R_{m+1}=k,$ and $R_i=k[x_i,\dots,x_m]$ for $i=1,\dots,m$. We now compute $J_i=I\cap R_i$. By Lemma \ref{GB} and \cite[Proposition 2.13]{St98} \[J_i=\big{\langle} G\cap R_i\big{\rangle},\quad i=1,\dots,m+1.\]
In particular,
\[J_i=\begin{cases}
\big{\langle} g_i,\dots, g_{m-1}\big{\rangle},\phantom{0}\quad \text{ if }i\in\{1,\dots, m-1\},\\
0, \phantom{\big{\langle} g_i,\dots, g_{m-1}\big{\rangle}}\quad \text{ if }i\in\{m, m+1\}.
\end{cases}\]

\begin{remark}\label{rem}
Note that $I\cap k[x_j]=0$ for every $j$. Indeed, if $I\cap k[x_j]\neq 0$ for some $j$, then, looking at the generators $g_i$ of $I$, we see that $I\cap k[x_j]\neq 0$ for {\em every} $j=1,\dots,m$. This contradicts $J_m=I\cap k[x_m]=0$.
\end{remark}
The algorithm requires studying, at every step, the (ir)reducibility of a polynomial $f$ over an appropriate field. We will need the following standard result.

\begin{lemma}\label{TeoK}
Let $F$ be an arbitrary field, $n\geq 1$, and $a\in F$. Then $x^n-a$ is irreducible over $F$ if and only if $a\not\in F^p$ for all primes $p$ dividing $n$ and $a\not\in -4F^4$ whenever $4\mid n$.
\end{lemma}
 \begin{proof}
 See \cite[Theorem 2.6 on page 425]{K89}.
 \end{proof}

\begin{proof}[Proof of Proposition \ref{irred}]
If the hypothesis on ${\bm a}$ is not satisfied, then $a_i+a_{i+1}+\dots+a_j=0$ for some $1\leq i\leq j\leq m-1$. Hence, in the system \eqref{system} defining our curve, we can replace the equation $g_i=0$ with $x_i^d-x_{j+1}^d=0$. Since the characteristic $q$ of the field $k$ does not divide $d$, the polynomial $x_i^d-x_j^d$ is the product of $d$ distinct irreducible factors. In particular, the curve $\mathcal{C}(d,{\bm a},q)$ is reducible.

Conversely, assume that the sums $A_i=\sum^{m-1}_{k= i}a_k$ are all distinct and non-zero. We prove that the ideal $\langle g_1,\dots,g_{m-1}\rangle$ is prime in $k[x_1,\dots,x_m]$. We argue by strong induction on $m$.

First, consider the case $m=2$. Following the algorithm, we obtain $f=0$ for $i=3$ . At the second (and last) iteration, for $i=2$, we have $f=x_1^d-x_2^d-a_1$. We claim that $x_1^d-x_2^d-a_1$ is irreducible over the field $k(x_2)$. Let $p$ be a prime, with $p \mid d$. Assume by contradiction that $x_2^d+a_1=\alpha^p$ for some $\alpha\in k(x_2)$. Let $\alpha=A/B$ for coprime polynomials $A,B\in k[x_2]$. Hence,
\begin{equation}\label{1}
\big(B(x_2)\big)^p(x_2^d+a_1)=\big(A(x_2)\big)^p.
\end{equation}
Let $\theta\in k$ be such that $\theta^{d}+a_1=0$. Recall that $a_1\neq 0$ by assumption. Hence, $\theta\neq 0$. It follows from \eqref{1} that $A(\theta)=0$. Taking a derivative on both sides of \eqref{1}, we obtain
\begin{equation}\label{2}
pB'(x_2)\big(B(x_2)\big)^{p-1}(x_2^d+a_1)+\big(B(x_2)\big)^p dx_2^{d-1}=pA'(x_2)\big(A(x_2)\big)^{p-1}.
\end{equation}
Evaluating \eqref{2} at $\theta$ yields
\begin{equation}\label{3}
\big(B(\theta)\big)^p d\theta^{d-1}=0.
\end{equation}
Since $\theta\neq 0$ and $d$ is not a multiple of the characteristic of $k$, \eqref{3} implies $B(\theta)=0$. We have now reached a contradiction, since $A$ and $B$ were taken to be coprime. Hence, $x_2^d+a_1\not\in \big(k(x_2)\big)^p$. By the same argument, $x_2^d+a_1\not\in -4\big(k(x_2)\big)^4$ if $4\mid d$. Lemma \ref{TeoK} then implies that  $x_1^d-x_2^d-a_1$ is irreducible over $k(x_2)$, and the proof of the case $m=2$ is complete.

Let now $m\geq 3$, and assume that $\langle g_1,\dots,g_{j-1}\rangle$ is a prime ideal in $k[x_1,\dots,x_{j}]$ for every $j\in\{1,\dots,m-1\}$. We want to prove that $\langle g_1,\dots,g_{m-1}\rangle$ is a prime ideal in $k[x_1,\dots,x_{m}]$. Running the primality algorithm, we see that the result follows if, for every $i=2,\dots,m$, the polynomial $f=x_{i-1}^d-x_i^d-a_{i-1}$ is irreducible over $k'$, where $k'$ is the quotient field of 
\begin{equation}\label{ultimata}
R'=\frac{k[x_i,\dots,x_m]}{\langle x_i^d-x_{i+1}^d-a_i,\dots,x_{m-1}^d-x_m^d-a_{m-1}\rangle}.
\end{equation}
Note that $R'$ is an integral domain by the inductive hypothesis.
To prove the irreducibility of $f$, we will use Lemma \ref{TeoK}. Let $p$ be a prime, with $p\mid d$. Assume by contradiction that $x_i^d+a_{i-1}=\alpha^p$, for some $\alpha\in k'$. Let $\bar{x}_j$ denote the equivalence class of $x_j$ in $k'$. Since the elements $\bar{x}_{i},\dots, \bar{x}_{m-1}$ are algebraic over the field $k(x_m)$, we have that $k'=k(x_m)[\bar{x}_{i},\dots, \bar{x}_{m-1}]$. We can thus find a representative of $\alpha\in k'$ which is a polynomial in the variables $x_{i},\dots, x_{m-1}$ with coefficients in $k(x_m)$. After clearing denominators, we obtain a representation of $\alpha$ as a quotient $A(x_i,\dots, x_m)/B(x_m)$, where $A\in k[x_i,\dots, x_m]$ and $B\in k[x_m]$. The equality $x_i^d+a_{i-1}=\alpha^p$ in $k'$ yields 
\begin{equation}\label{rri}
\big(B(x_m)\big)^p(x_i^d+a_{i-1})-\big(A(x_i,\dots, x_{m})\big)^p\in \langle x_i^d-x_{i+1}^d-a_i,\dots,x_{m-1}^d-x_m^d-a_{m-1}\rangle.
\end{equation}
We can assume without loss of generality that $A$ is of degree at most $d-1$ in the variables $x_{i+1},\dots, x_m$. Hence,
\begin{equation}\label{sos}
\big(B(x_m)\big)^p(x_i^d+a_{i-1})-\Bigg(\sum_{1\leq j_{i+1},\dots,j_{m}\leq d} c_{j_{i+1}\dots j_{m}}(x_i)\,x^{j_{i+1}}_{i+1}\cdots x_m^{j_m}\Bigg)^p=\sum_{j=i}^{m-1} f_j\cdot (x_j^d-x_{j+1}^d-a_j)
\end{equation}
for some polynomials $f_j\in k[x_i,\dots, x_m]$ and $ c_{j_{i+1}\dots j_{m}}\in k[x_i]$. Now let $(\theta_i,\dots, \theta_m)\in k^m$ satisfying
\begin{equation}\label{theta}
\begin{cases}
\theta_i^d=-a_{i-1},\\
\theta_{i+1}^d=-a_{i-1}-a_i,\\
\phantom{\theta}\vdots\\
\theta^d_{m-1}=-a_{i-1}-a_i-\dots -a_{m-2},\\
\theta^d_{m}=-a_{i-1}-a_i-\dots -a_{m-2}-a_{m-1}.\\
\end{cases}
\end{equation}
By our hypotheses on ${\bm a}$, the sums of the $a_j$ appearing in \eqref{theta} are never equal to zero. Moreover, since $d$ is not a multiple of the characteristic of the base field $k$, there are $d$ distinct values for every $\theta_j$. Substituting $\theta_i,\dots,\theta_{m}\in k$ satisfying \eqref{theta} in \eqref{sos} gives $A(\theta_i,\dots,\theta_m)=0$. Note that $A(\theta_i,\dots,\theta_{m-1}, x_m)$ is a polynomial of degree at most $d-1$ in $x_m$ with $d$ distinct roots in $k$. Hence each of its coefficients must be equal to zero. Fixing $\theta_i,\dots,\theta_{m-2}$ in each of those coefficients, we obtain polynomials of degree at most $d-1$ in $x_{m-1}$ with $d$ roots in $k$. Inductively, we conclude that every polynomial $c_{j_{i+1}\dots j_{m}}(x_i)$ appearing in \eqref{sos} must vanish for all $\theta_i$ such that $\theta_i^d=-a_{i-1}$. It follows that $A(x_i,\dots, x_{m})=(x_i^d+a_{i-1})V(x_i,\dots, x_m)$ for some $V\in k[x_i,\dots, x_m]$. From \eqref{rri} we thus get
\begin{equation*}
(x_i^d+a_{i-1})\Bigg(\big(B(x_m)\big)^p-\big(V(x_i,\dots, x_{m})\big)^p(x_i^d+a_{i-1})^{p-1}\Bigg)\in \langle x_i^d-x_{i+1}^d-a_i,\dots,x_{m-1}^d-x_m^d-a_{m-1}\rangle.
\end{equation*}
The ideal $\langle x_i^d-x_{i+1}^d-a_i,\dots,x_{m-1}^d-x_m^d-a_{m-1}\rangle$ is prime in $k[x_i,\dots, x_m]$ by the inductive hypothesis. Remark \ref{rem} then implies
\begin{equation}\label{la}
\big(B(x_m)\big)^p-\big(V(x_i,\dots, x_{m})\big)^p(x_i^d+a_{i-1})^{p-1}=\sum_{j=i}^{m-1} h_j\cdot (x_j^d-x_{j+1}^d-a_j)
\end{equation}
for some polynomials $h_j\in k[x_i,\dots, x_m]$. Substituting in \eqref{la} elements $\theta_1,\dots, \theta_m\in k$ that satisfy \eqref{theta}, we obtain $B(\theta_m)=0$ for every $\theta_m$ such that $\theta_m^d=-a_{i-1}-\dots - a_{m-1}$. Therefore \[B(x_m)=W(x_m)(x_m^d+a_{i-1}+\dots + a_{m-1})\]
for some polynomial $W\in k[x_m]$. Note that $x_m^d+a_{i-1}+\dots + a_{m-1}=x_i^d+a_{i-1}$ in $k'$, and therefore the elements $A$ and $B$ representing $\alpha$ have a common factor in $R'$. We can thus remove this common factor and write $\alpha=\widetilde{A}(x_i,\dots, x_m)/\widetilde{B}(x_m)$, where $\widetilde{A}\in k[x_i,\dots, x_m]$ and $\widetilde{B}\in k[x_m]$ are polynomials whose degrees satisfy $\deg \widetilde{A}=\deg A -d$, $\deg \widetilde{B}=\deg B -d$. Repeating the reasoning above with $A$ and $B$ replaced by $\widetilde{A}$ and $\widetilde{B}$ respectively, we deduce that $\widetilde{A}$ and $\widetilde{B}$ again share a common factor in $R'$. We remove that common factor, and repeat the same argument once more. After finitely many steps, we reach a contradiction.
We have thus proved that $x_i^d+a_{i-1}\not\in (k')^p$. With the same argument, one can show that $x_i^d+a_{i-1}\not\in -4(k')^4$ if $4\mid d$. Lemma \ref{TeoK} then implies that $f=x_{i-1}^d-x_i^d-a_{i-1}$ is irreducible over $k'$, as wanted.
\end{proof}

\section{Point count on the curve}\label{SectionCount}
Let $b,d, q$ be positive integers, with $d\geq 2$ and $(b,q)=1$. Let ${\bm a}=(a_1,\dots, a_{m-1})\in {\bm Z}_q^{m-1}$. Define
\begin{equation}\label{defnu}
\nu(d,{\bm a},q)=\#\big{\{}{\bm x}\in{\bm Z}_q^{m}\,\colon\, bx_i^d-bx_{i+1}^d=a_i\, (\text{mod}\, q), 1\leq i\leq m-1\big{\}}.
\end{equation}

\begin{remark}
When $q$ is a prime, the quantity $\nu(d,{\bm a},q)$ corresponds to the number of points in ${\bm Z}_q^{m}$ of the curve $\mathcal{C}(d,{\bm a}\bar{b},q)$, following the notation of Section \ref{SectionCurve}. Here $\bar{b}\in{\bm Z}_q$ is such that $b\bar{b}=1$, and ${\bm a}\bar{b}$ is the element $(a_1\bar{b},\dots, a_{m-1}\bar{b})\in {\bm Z}_q^{m-1}$ .
\end{remark}

\begin{notation}
From now on we use the notation $\mathcal{C}(d,{\bm a}\bar{b},q)$ even when $q$ is not prime (but still $(b,q)=1$) to indicate the curve consisting of the points $(x_1,\dots, x_m)$ satisfying the equations $bx_i^d-bx_{i+1}^d=a_i\, (\text{mod}\, q)$ for $1\leq i\leq m-1$.
\end{notation}

Let $q=\prod_{j=1}^r p_j^{e_j}$ be a prime factorization of $q$. By the Chinese Remainder Theorem,
\begin{equation}\label{CRT}
\nu(d,{\bm a},q)=\prod_{j=1}^r \nu(d,{\bm a},p_j^{e_j}).
\end{equation}
For every $j\in\{1,\dots, r\}$, we define the number $A(d,{\bm a},p_j^{e_j})$ by
\begin{equation*}
A(d,{\bm a},p_j^{e_j})=\nu(d,{\bm a}, p_j^{e_j})-p_j^{e_j}.
\end{equation*}
Equation \eqref{CRT} yields
\begin{equation}\label{rew}
\nu(d,{\bm a}, q)=\prod_{j=1}^r(p_j^{e_j}+A(d,{\bm a}, p_j^{e_j}))=q\prod_{j=1}^r \bigg(1+\frac{A(d,{\bm a},p_j^{e_j})}{p_j^{e_j}}\bigg)=q\sum_{\mathcal{S}\subseteq\{1,\dots,r\}}\frac{A(d,{\bm a}, c_{\mathcal{S}})}{c_{\mathcal{S}}},
\end{equation}
where for each non-empty subset $\mathcal{S}\subseteq\{1,\dots,r\}$ we let
\begin{equation*}
c_{\mathcal{S}}=\prod_{j\in\mathcal{S}} p_j^{e_j},\quad\quad\quad A(d,{\bm a}, c_{\mathcal{S}})=\prod_{j\in\mathcal{S}} A(d,{\bm a}, p_j^{e_j}),
\end{equation*}
and $c_{\mathcal{S}}=1$ for $\mathcal{S}=\emptyset$.

In the next lemma we will establish a bound for $|A(d,{\bm a},c_{\mathcal{S}})|$. Following \cite{KR99}, for a prime $p$ and ${\bm a}=(a_1,\dots, a_{m-1})$, we denote by $r_{\textit{eff}}({\bm a},p)$ the number of distinct components of an element ${\bm y}=(y_1,\dots, y_m)$ satisfying
\begin{equation}\label{5.8}
y_i-y_{i+1}=a_i\,\,(\text{mod } p),\quad\quad 1\leq i\leq m-1.
\end{equation}
This number is well defined, independent of the particular solution. For $1\leq i<j\leq m$, let
\begin{equation*}
\sigma_{ij}({\bm a})=\sum_{k=i}^{j-1}a_k,
\end{equation*}
 so that $\sigma_{i,i+1}({\bm a})=a_i, \sigma_{ij}=\sum_{k=i}^{j-1}\sigma_{k,k+1}$. Let $D({\bm a})=\prod_{1\leq i\leq j\leq m}\sigma_{ij}({\bm a})$.  A solution of the system \eqref{5.8} has distinct components (that is, $r_{\textit{eff}}({\bm a},p)=m$) if and only if $p$ does not divide $D({\bm a})$, since $y_i-y_j=\sum_{k=i}^{j-1}(y_k-y_{k+1})=\sum_{k=i}^{j-1}a_k=\sigma_{ij}({\bm a})$.

\begin{lemma}\label{estimateonA} Let $\tilde{q}$ be the squarefree part of $q$. Assume that there exists $\delta>0$ such that $\tilde{q}\geq q^{1-\delta}$. Then, for every subset $\mathcal{S}\subset\{1,\dots, r\}$, we have
\begin{equation*}|A(d,{\bm a},c_{\mathcal{S}})|\ll_{m,d,\varepsilon} c_{\mathcal{S}}^{1/2+\varepsilon}(c_{\mathcal{S}}, D({\bm a}))^{1/2}q^{2\delta m}.
\end{equation*}
\end{lemma}

\begin{proof}
First we consider the case of $c=p$ prime with $p\nmid d$. Applying the Riemann Hypothesis for curves over finite fields (see \cite{W48} and \cite[Theorem 5A and Corollary 5B]{Sc76}) as in \cite[Proposition 4]{KR99}, we obtain
 \begin{equation}\label{estimate1}
 \nu(d,{\bm a},p)=d^{m-r_{\textit{eff}}({\bm a},p)}(p+B({\bm a},p)),\quad\text{ with }|B({\bm a},p)|\ll_m p^{1/2}.
 \end{equation}
 We use estimate \eqref{estimate1} for the primes $p_j$ such that $e_j=1$ and $p_j\nmid d$. For the other primes, we use the trivial estimate
 \begin{equation}\label{estimate2}
 |A(d,{\bm a}, p_j^{e_j})|\leq p_j^{e_jm}.
 \end{equation}
 Note that since $\tilde{q}\geq q^{1-\delta}$ it follows that
\begin{equation}\label{til}
\prod_{\substack{j\in \mathcal{S}\\ e_j\geq 2}} p_j^{e_j}\leq q^{2\delta m}.
 \end{equation}
Multiplying the inequalities in \eqref{estimate1} and \eqref{estimate2}, using \eqref{til}, and recalling that $\#\mathcal{S}\ll_{m,d,\varepsilon}c_{\mathcal{S}}^{\varepsilon}$, we obtain
\begin{equation*}
\begin{split}
|A(d,{\bm a}, c_{\mathcal{S}})|&=\prod_{\substack{j\in \mathcal{S}\\ e_j=1,\, p_j\nmid d}}|A(d,{\bm a}, p_j)|\prod_{\substack{j\in \mathcal{S}\\ e_j\geq 2 \text{ or } p_j\mid d}}|A(d,{\bm a}, p_j^{e_j})|\\ &\leq \prod_{\substack{j\in \mathcal{S}\\ e_j=1,\, p_j\nmid d}} C_{m, d} \,p_j^{1/2}(p_j, D({\bm a}))^{1/2}\prod_{\substack{j\in \mathcal{S}\\ e_j\geq 2}} p_j^{e_jm}\prod_{\substack{j\in \mathcal{S}\\e_j=1,\, p_j\mid d}} p_j^{e_jm}.\\
 &\leq C_{m,d,\varepsilon} \,c_{\mathcal{S}}^{1/2+\varepsilon}(c_{\mathcal{S}}, D({\bm a}))^{1/2}q^{2\delta m},\\
\end{split}
\end{equation*}
which concludes the proof of the lemma.
\end{proof}

\begin{lemma}\label{zerosum}
For every $c>1$,
\begin{equation*}
\sum_{{\bm a}\,(\moda  c)}A(d,{\bm a}, c)=0.
\end{equation*}
\end{lemma}

\begin{proof}
For every prime $p$,
\begin{equation}\label{somma1}
\sum_{{\bm a}\,(\moda  p)}\nu(d,{\bm a},p)=\sum_{{\bm x}\,(\moda  p)}\,\sum_{\substack{{\bm a}\,(\moda  p)\\{\bm x}\in \mathcal{C}(d,{\bm a}\bar{b},p)}} 1 = p^m.
\end{equation}
By definition of $A(d,{\bm a},p)$,
\begin{equation}\label{somma2}
\sum_{{\bm a}\,(\moda  p)}\nu(d,{\bm a},p)=\sum_{{\bm a}\,(\moda  p)}(p+A(d,{\bm a},p))=p^m+\sum_{{\bm a}\,(\moda  p)} A(d,{\bm a},p).
\end{equation}
Combining \eqref{somma1} and \eqref{somma2} we obtain
\begin{equation*}
\sum_{{\bm a}\,(\moda  p)}A(d,{\bm a}, p)=0.
\end{equation*}
The result now follows from the multiplicativity of the above sums.
\end{proof}


\section{A convergence theorem}
The goal of this section is proving the following result.
\begin{theorem}\label{T2}
Let $d\geq 2$. Fix $m\geq 2$ and $0<\delta_0<\frac{1}{4m}$. There exists $\delta=\delta(m,\delta_0)>0$ such that for every $f\in C^{\infty}_c({\bf R}^{m-1})$ one has
\begin{equation*}
R^{(m)}\bigg(N,d,\frac{b}{q},f\bigg)\rightarrow\int_{{\bf R}^{m-1}} f({\bf x})\, d{\bf x}
\end{equation*}
uniformly for $(b,q)=1, q^{1-1/(2m)+\delta_0}\leq N\leq q^{1-\delta_0}$, as $q\rightarrow\infty$ such that $\tilde{q}\geq q^{1-\delta}$.
\end{theorem}

Following \cite[Section 4]{Z03}, we reduce the proof of Theorem \ref{T2} to a point count on a curve over a finite field. We then prove Theorem \ref{T2} using the results of Sections \ref{SectionCurve} and \ref{SectionCount}.

By approximating $f$ from above and below by step functions, it is enough to prove the statement for the characteristic function of a compact set with piecewise smooth boundary $I\subset {\bf R}^{m-1}$. Given $b$ and $q$ as in the hypotheses, we thus want to show, for $q\to \infty$, that
\begin{equation}\label{goal}
R^{(m)}(N, d,b/q,I)\to \text{vol}(I), 
\end{equation}
 where $N R^{(m)}(N, d,b/q,I)$ is the number of $m$-tuples $(x_1,\dots, x_m)$ with distinct components $x_1,\dots, x_m$ in $\{1,\dots, N\}$ such that
\begin{equation*}
N\bigg(\bigg{\{}\frac{bx_1^d}{q}\bigg{\}}-\bigg{\{}\frac{bx_2^d}{q}\bigg{\}},\dots, \bigg{\{}\frac{bx_{m-1}^d}{q}\bigg{\}}-\bigg{\{}\frac{bx_m^d}{q}\bigg{\}}\bigg)\in I.
\end{equation*}
We can write $R^{(m)}(N, d,b/q,I)$ in the form
\begin{equation}\label{star}
R^{(m)}(N, d,b/q,I)=\frac{1}{N}\sum_{{\bm a}\in sI}^{*}\nu(N,d,{\bm a},q),
\end{equation}
where $s=q/N$ is the dilate factor and
\begin{equation*}
\nu(N,d,{\bm a},q)=\#\{1\leq x_i\leq N\,\colon\, bx_i^d-bx_{i+1}^d=a_i\,(\text{mod } q), 1\leq i\leq m-1\}.
\end{equation*}
Here the star in \eqref{star} denotes summation over the vectors ${\bm a}$ for which the partial sums $A_i=\sum_{k\geq i}a_k$ are all distinct and non-zero, a condition which comes from the requirement that the $m$-tuples $(x_1,\dots, x_m)$ to be counted in $R^{(m)}(N, d,b/q,I)$ have distinct components.

\begin{lemma} We have
\begin{equation}\label{5.1}
R^{(m)}(N, d,b/q,I)=\frac{1}{Nq^m}\sum_{{\bm a}\in sI}^*\,\sum_{{\bm r} (\text{mod }q)}\sum_{{\bm y}\in \mathcal{C}(d,{\bm a}\bar{b},q)}e\bigg(\frac{-{\bm r}\cdot {\bm y}}{q}\bigg)\prod_{i=1}^m\,\sum_{1\leq x_i\leq N}e\bigg(\frac{r_ix_i}{q}\bigg).
\end{equation}
\end{lemma}
\begin{proof}
The lemma can be deduced from a standard application of Fourier expansion. See \cite[Section 4]{Z03} for details.
\end{proof}

The last sums appearing in \eqref{5.1} are geometric progressions and can be bounded by
\begin{equation}\label{ebound}
\norm{\sum_{1\leq x_i\leq N}e\bigg(\frac{r_ix_i}{q}\bigg)}\ll \min\bigg{\{} N, \frac{q}{|r_i|}\bigg{\}},
\end{equation}
where the $r_i$ are assumed to lie in the interval $\big{[}-\frac{q}{2},\frac{q}{2}\big{]}$
To prove \eqref{goal}, we first consider the contribution of ${\bm r}=0$ on the right side of \eqref{5.1}. We obtain the term
\begin{equation*}
\mathcal{M}=\frac{N^{m-1}}{q^m}\sum_{{\bm a}\in sI}^*\nu(d, {\bm a}, q),
\end{equation*}
where $\nu(d, {\bm a}, q)$ is the quantity defined in \eqref{defnu}. We let
\begin{equation}\label{remainder}
\mathcal{E}=R^{(m)}(N, d,b/q,I)-\mathcal{M}.
\end{equation}
Theorem \ref{T2} follows from the two lemmas below.
\begin{lemma}\label{LemmaMain}
As $q\to \infty$, we have $\mathcal{M}=\vol(I)+o(1).$
\end{lemma}
\begin{lemma}\label{LemmaRemainder}
As $q\to \infty$, we have $\mathcal{E}=o(1).$
\end{lemma}
\begin{proof}[Proof of Lemma \ref{LemmaMain}]
By \eqref{rew} we can rewrite $\mathcal{M}$ as
\begin{equation*}
\mathcal{M}=\frac{N^{m-1}}{q^{m-1}}\sum_{\mathcal{S}\subset\{1,\dots,r\}}\frac{1}{c_{\mathcal{S}}}\sum_{{\bm a}\in sI}^*A(d,{\bm a}, c_{\mathcal{S}}).
\end{equation*}
Recall that by definition $c_{\mathcal{S}}=1$ when $\mathcal{S}=\emptyset$. The contribution of $\mathcal{S}=\emptyset$ is therefore
\begin{equation*}
\bigg(\frac{N}{q}\bigg)^{m-1}\sum_{{\bm a}\in sI}^*1=\frac{1}{s^{m-1}}\#\big{\{}(sI)^*\cap {\bm Z}^{m-1}\big{\}}=\vol(I)+O\bigg(\frac{1}{s}\bigg),
\end{equation*}
where $(sI)^*$ is the set of vectors ${\bm a}\in sI$ with all the partial sums $A_i$ distinct and non-zero.

For the remaining divisors $c_{\mathcal{S}}$ of $q$ we distinguish two cases. Let $\mathfrak{A}$ be the set of all subsets $\mathcal{S}\subset\{1,\dots,r\}$ for which $c_{\mathcal{S}}>s_1$, and let $\mathfrak{B}$ be the set of all $\mathcal{S}\subset\{1,\dots,r\}$ for which $1<c_{\mathcal{S}}<s_1$. Here $s_1$ is a parameter that will be chosen later. By Lemma \ref{estimateonA},
\begin{equation}\label{f1}
\begin{split}
\frac{1}{s^{m-1}}\sum_{\mathcal{S}\in \mathfrak{A}} \frac{1}{c_{\mathcal{S}}}\sum_{{\bm a}\in sI}^* A(d,{\bm a},c_{\mathcal{S}})&\ll_{m,d,\varepsilon} \frac{1}{s^{m-1}}\sum_{\mathcal{S}\in \mathfrak{A}} \frac{1}{c_{\mathcal{S}}} c_{\mathcal{S}}^{1/2+\varepsilon}\sum_{{\bm a}\in sI}(c_{\mathcal{S}}, D({\bm a}))^{1/2}q^{2\delta m}\\
&= \frac{1}{s^{m-1}}\sum_{\mathcal{S}\in \mathfrak{A}} c_{\mathcal{S}}^{\varepsilon-1/2}q^{2\delta m}\sum_{t\mid c_{\mathcal{S}}}  t^{1/2}\#\{{\bm a}\in sI\,\colon\, (c_{\mathcal{S}}, D({\bm a}))=t\}\\&\leq \frac{1}{s^{m-1}}\sum_{\mathcal{S}\in\mathfrak{A}}c_{\mathcal{S}}^{\varepsilon-1/2} q^{2\delta m}\sum_{t\mid c_{\mathcal{S}}} t^{1/2}\#(J(t)),
\end{split}
\end{equation}
where $J(t)=\{{\bm a}\in sI\,\colon\, t\mid D({\bm a})\}$. One can easily see that
\begin{equation}\label{f2}
\#(J(t))\ll_{m,d,\varepsilon,I} t^{\varepsilon}s^{m-1}\bigg(\frac{1}{t}+\frac{1}{s}\bigg).
\end{equation}
Equations \eqref{f1} and \eqref{f2} yield
\begin{equation*}
\frac{1}{s^{m-1}}\sum_{\mathcal{S}\in \mathfrak{A}} \frac{1}{c_{\mathcal{S}}}\sum_{{\bm a}\in sI}^* A(d,{\bm a},c_{\mathcal{S}})\ll_{m,d,\varepsilon,I}\sum_{\mathcal{S}\in \mathfrak{A}} c_{\mathcal{S}}^{\varepsilon-1/2}\sum_{t\mid c_{\mathcal{S}}}q^{2\delta m}t^{1/2+\varepsilon}\bigg(\frac{1}{t}+\frac{1}{s}\bigg)=\sigma_1+\sigma_2,
\end{equation*}
where
\begin{equation}\label{sigma1}
\sigma_1=\sum_{\mathcal{S}\in \mathfrak{A}} c_{\mathcal{S}}^{\varepsilon-1/2}\sum_{t\mid c_{\mathcal{S}}}q^{2\delta m}t^{-1/2+\varepsilon}\leq q^{2\delta m} s_1^{\varepsilon-1/2}\#\{(c_{\mathcal{S}},t)\,\colon\, t\mid c_{\mathcal{S}}\mid q\}\leq q^{2\delta m+2\varepsilon}s_1^{\varepsilon-1/2},
\end{equation}
\begin{equation}\label{sigma2}
\sigma_2=\sum_{\mathcal{S}\in \mathfrak{A}} c_{\mathcal{S}}^{\varepsilon-1/2}\sum_{t\mid c_{\mathcal{S}}}q^{2\delta m}t^{1/2+\varepsilon}\frac{1}{s}\leq \frac{1}{s}\sum_{c_{\mathcal{S}}\mid q}\sum_{t\mid c_{\mathcal{S}}}q^{2\delta m}(tc_{\mathcal{S}})^{\varepsilon}\bigg(\frac{t}{c_{\mathcal{S}}}\bigg)^{1/2}\leq \frac{1}{s}\,q^{4\varepsilon+2\delta m}.
\end{equation}
Equation \eqref{sigma2} shows that $\sigma_2 = o(1)$ as $q\to\infty$ for $\delta$ small enough in terms of $m$. Letting $s_1=\sqrt{s}$, we see from \eqref{sigma1} that \begin{equation}\label{sigma12}\sigma_1\leq s^{\varepsilon/2-1/4}q^{2(\delta m+\varepsilon)}.
\end{equation} Since $s=q/N\geq q^{\delta_0}$, the inequality \eqref{sigma12} implies that $\sigma_1$ is $o(1)$ as $q\to\infty$ for $\delta$ small enough in terms of $m$ and $\delta_0$.

We now consider the divisors $c_{\mathcal{S}}$ of $q$ such that $1< c_{\mathcal{S}}<s_1=\sqrt{s}$. Recall that we denote by $(sI)^*$
the set of vectors ${\bm a}\in sI$ such that all the partial sums $A_i$ are distinct and non-zero. We divide the region $(sI)^*$ into integer cubes of side $c_{\mathcal{S}}$ of the form $y+c_{\mathcal{S}}B$, where $y\in c_{\mathcal{S}}{\bm Z}^{m-1}$ and $B=\{0\leq x_i< 1\}$ is the unit cube in ${\bm R}^{m-1}$. We call a cube {\em $c_{\mathcal{S}}$-interior} if it is entirely contained in $(sI)^*$. By the Lipschitz principle (see \cite{D51}) it follows that the number $n_{c_{\mathcal{S}}}$ of {\em $c_{\mathcal{S}}$-interior} cubes is given by
\begin{equation*}
n_{c_{\mathcal{S}}}=\vol\Big(\frac{s}{c_{\mathcal{S}}} I\Big)+O_{I}\bigg(\Big(\frac{s}{c_{\mathcal{S}}}\Big)^{m-2}\bigg)=\Big(\frac{s}{c_{\mathcal{S}}}\Big)^{m-1}\vol(I)+O_I\bigg(\Big(\frac{s}{c_{\mathcal{S}}}\Big)^{m-2}\bigg).
\end{equation*}
We say that a point $a\in sI\cap {\bm Z}^{m-1}$ is {\em $c_{\mathcal{S}}$-interior} if it is contained in a {\em $c_{\mathcal{S}}$-interior} cube, and is {\em $c_{\mathcal{S}}$-boundary} otherwise. Each interior cube contains $c_{\mathcal{S}}^{m-1}$
{\em $c_{\mathcal{S}}$-interior} points, so the total number of {\em $c_{\mathcal{S}}$-interior} points is
\begin{equation}\label{cinterior}
c_{\mathcal{S}}^{m-1}n_{c_{\mathcal{S}}}=s^{m-1}\vol(I)+O(c_{\mathcal{S}}s^{m-2}).
\end{equation}
The total number of points of $sI\cap {\bm Z}^{m-1}$ is $s^{m-1}\vol(I)+O(s^{m-2})$. Subtracting the number of {\em $c_{\mathcal{S}}$-interior} points given by \eqref{cinterior}, we obtain that the number of {\em $c_{\mathcal{S}}$-boundary} points is $O(c_{\mathcal{S}}s^{m-2})$.

For every divisor $c_{\mathcal{S}}$ of $q$, with $\mathcal{S}\in\mathfrak{B}$, we write
\begin{equation*}
\sum_{{\bm a}\in sI}^*A(d,{\bm a}, c_{\mathcal{S}})=\sum_{{\bm a} \, c_{\mathcal{S}}\textit{-boundary}}^*A(d,{\bm a}, c_{\mathcal{S}})\,+ \sum_{{\bm a} \, c_{\mathcal{S}}\textit{-interior}}^*A(d,{\bm a}, c_{\mathcal{S}}).
\end{equation*}
Since the sum over each {\em $c_{\mathcal{S}}$-interior} cube is just a sum over $({\bm Z}/c_{\mathcal{S}}{\bm Z})^{m-1}$, Lemma \ref{zerosum} implies
\begin{equation*}
\sum_{{\bm a}\in sI}^*A(d,{\bm a}, c_{\mathcal{S}})=\sum_{{\bm a} \, c_{\mathcal{S}}\textit{-boundary}}^*A(d,{\bm a}, c_{\mathcal{S}}).
\end{equation*}
By Lemma \ref{estimateonA} and the fact that the number of {\em $c_{\mathcal{S}}$-boundary} points is $O(c_{\mathcal{S}}s^{m-2})$ we obtain

\begin{equation*}
\begin{split}
\frac{1}{s^{m-1}}\sum_{\mathcal{S}\in \mathfrak{B}} \frac{1}{c_{\mathcal{S}}}\sum_{{\bm a}\in sI}^* A(d,{\bm a},c_{\mathcal{S}})&=
\frac{1}{s^{m-1}}\sum_{\mathcal{S}\in \mathfrak{B}} \frac{1}{c_{\mathcal{S}}}\,\sum_{{\bm a}\, c_{\mathcal{S}}\textit{-boundary}}^* A(d,{\bm a},c_{\mathcal{S}})\\&\ll_{d,m,\varepsilon} \frac{1}{s^{m-1}}\sum_{\mathcal{S}\in\mathfrak{B}}\frac{1}{c_{\mathcal{S}}} c_{\mathcal{S}}^{2+\varepsilon}s^{m-2}q^{2\delta m}=\frac{1}{s}\sum_{\mathcal{S}\in\mathfrak{B}}c_{\mathcal{S}}^{1+\varepsilon}q^{2\delta m}\\&\ll_{\varepsilon}q^{2\delta m}\frac{1}{s}s_1^{1+\varepsilon}\#\{c_{\mathcal{S}}\mid q\}\ll_{\varepsilon} s^{-1/2+\varepsilon/2} q^{\varepsilon+2\delta m}\leq q^{-\delta_0/2+\varepsilon\delta_0/2+\varepsilon+2\delta m},
\end{split}
\end{equation*}
which is $o(1)$ as $q\to\infty$ for $\delta$ small enough in terms of $m$ and $\delta_0$. This concludes the proof of Lemma \ref{LemmaMain}.
\end{proof}

\begin{proof}[Proof of Lemma \ref{LemmaRemainder}]
By \eqref{5.1}, \eqref{ebound} and \eqref{remainder},   
\begin{align}\label{remainder2}
\mathcal{E}&=\frac{1}{Nq^m}\sum_{{\bm a}\in sI}^*\,\sum_{\substack{{\bm r} (\text{mod }q)\\{\bm r} \neq 0}}\sum_{{\bm y}\in \mathcal{C}(d,{\bm a}\bar{b},q)}e\bigg(\frac{-{\bm r}\cdot {\bm y}}{q}\bigg)\prod_{i=1}^m\,\sum_{1\leq x_i\leq N}e\bigg(\frac{r_ix_i}{q}\bigg) \notag\\
&\ll\frac{1}{Nq^m}\sum_{{\bm a}\in sI}^*\,\sum_{\substack{{\bm r} (\text{mod }q)\\{\bm r} \neq 0}}\sum_{{\bm y}\in \mathcal{C}(d,{\bm a}\bar{b},q)}e\bigg(\frac{-{\bm r}\cdot {\bm y}}{q}\bigg)\prod_{i=1}^m\,\min\bigg{\{} N, \frac{q}{|r_i|}\bigg{\}}.
\end{align}
We start by observing that, if $q=\prod_jp_j^{k_j}$ is the decomposition of $q$ into primes, then (see \cite{Z03} for details)
\begin{equation}\label{rema}
\sum_{{\bm y}\in \mathcal{C}(d,{\bm a}\bar{b},q)}e\bigg(-\frac{{\bm r}\cdot {\bm y}}{q}\bigg)=\prod_j \sum_{{\bm y}\in \mathcal{C}(d,{\bm a}\bar{b},p_j^{k_j})}e\bigg(-\frac{b_j{\bm r}\cdot {\bm y}}{p_j^{k_j}}\bigg),
\end{equation}
where the $b_j$ are given by
\begin{equation*}
b_j=\prod_{l\neq j}p_l^{-k_l}\, (\text{mod } p_j^{k_j}).
\end{equation*}
We use the trivial bound
\begin{equation*}
 \Bigg{|}\sum_{{\bm y}\in \mathcal{C}(d,{\bm a}\bar{b},p_j^{k_j})}e\bigg(-\frac{b_j{\bm r}\cdot {\bm y}}{p_j^{k_j}}\bigg)\Bigg{|}\leq p_j^{mk_j}
\end{equation*}
for the factors on the right side of \eqref{rema} for which $k_j\geq 2$. Since $\tilde{q}\leq q^{1-\delta}$, we obtain
\begin{equation}\label{kj2}
 \Bigg{|}\prod_{k_j\geq 2}\sum_{{\bm y}\in \mathcal{C}(d,{\bm a}\bar{b},p_j^{k_j})}e\bigg(-\frac{b_j{\bm r}\cdot {\bm y}}{p_j^{k_j}}\bigg)\Bigg{|}\leq
 \prod_{k_j\geq 2}p_j^{mk_j}\leq q^{2\delta m}.
\end{equation}
We consider those primes $p_j$ for which $k_j=1.$ For such primes $p_j$ which divide $d,$ we have
the trivial bound
\begin{equation}\label{pjd}
 \Bigg{|}\prod_{\substack{k_j=1\\p_j\mid d}}\sum_{{\bm y}\in \mathcal{C}(d,{\bm a}\bar{b},p_j)}e\bigg(-\frac{b_j{\bm r}\cdot {\bm y}}{p_j}\bigg)\Bigg{|}\leq d^m.
\end{equation}
Next, for primes $p_j\nmid d,$  we use the Bombieri-Weil inequality \cite[Theorem 6]{B66}, which gives
\begin{equation}\label{BW}
\Bigg{|}\sum_{{\bm y}\in \mathcal{C}(d,{\bm a}\bar{b},p_j)}e\bigg(-\frac{b_j{\bm r}\cdot {\bm y}}{p_j}\bigg)\Bigg{|}\ll_{m} p_j^{1/2}
\end{equation}
provided that  the partial sums $A_i$ are distinct mod $p_j.$
We can apply \eqref{BW} only if ${\bm y}\cdot {\bm r}$ is not constant on any component of the curve $\mathcal{C}(d,{\bm a}\bar{b},p_j).$ Note that Proposition \ref{irred} guarantees that every curve $\mathcal{C}(d,{\bm a}\bar{b},p_j)$ is irreducible.
 In the next paragraph we prove that if ${\bm y}\cdot {\bm r}$ is constant on a curve $\mathcal{C}(d,{\bm a}\bar{b},p_j)$ then ${\bm r}=0$, which is never the case for the terms considered in Lemma \ref{LemmaRemainder} (see \eqref{remainder2}).
 
 Let $k=\overline{\bm Z}_{p_j}$ denote the algebraic closure of the field ${\bm Z}_{p_j}={\bm Z}/p_j{\bm Z}$. Then, in the function field $k(Y_1,\dots, Y_m)$ of the curve $\mathcal{C}(d,{\bm a}\bar{b},p_j)$, $Y_1$ is a variable and $Y_2,\dots, Y_m$ are algebraic functions such that
\begin{equation*}
Y^d_i=Y^d_1 - (a_1+\dots +a_{i-1})\bar{b}\,\,\,\,\,\,\,\,\,\text{ for } \,\,2\leq i\leq m.
\end{equation*}
Recall that the key step in the proof of Proposition \ref{irred} was to show the irreducibility of some polynomials (there denoted by $f$) over some quotient fields (see \eqref{ultimata}). The same exact argument shows, for every $i\in\{2,\dots,m\}$, that the polynomial
\begin{equation}\label{polio}
x^d-(Y_1^d-(a_1+\dots+a_{i-1})\bar{b})
\end{equation}
is irreducible in the ring $k(Y_1,\dots, Y_{i-1})[x]$. It follows that 
 \begin{equation}\label{unooo}
 [k(Y_1,\dots, Y_m):k(Y_1)]=d^{m-1}.
 \end{equation}
Assume now by contradiction that ${\bm y}\cdot {\bm r}=c$, with $c\in k$ and ${\bm r}\neq 0$. Let $j_0\in\{1,\dots, m\}$ be such that $r_{j_0}\neq 0$. Then $Y_{j_0}$ belongs to $k(Y_1,\dots, Y_{j_0-1}, Y_{j_0+1},\dots, Y_m)$, and therefore
\begin{equation}\label{equality}
k(Y_1,\dots, Y_m)=k(Y_1,\dots,Y_{j_0-1}, Y_{j_0+1},\dots, Y_m).
\end{equation}
Applying again the irreducibility of the polynomials in \eqref{polio} we obtain
\[  [k(Y_1,\dots,Y_{j_0-1}, Y_{j_0+1},\dots, Y_m) \colon k(Y_1)]=d^{m-2},\]
which, together with \eqref{equality}, contradicts \eqref{unooo}.
We thus conclude that for the primes $p_j\nmid d$ such that the sums $A_i$ are distinct mod $p_j,$
the inequality \eqref{BW} holds true.

Now, in general, for each pair of ${\bm a}$ and $p_j,$ we have a partition $\mathscr P=\{V_1, \ldots, V_{\ell}\}$ of the set $V=\{1, 2, \ldots, m\}$ where $A_{i_1}=A_{i_2} \pmod{p_j}$ if and only if $i_1, i_2 \in V_{\ell'}$ for some $1\leq \ell' \leq \ell.$
Using this partition, for each ${\bm r}$ we write ${\bm r}\cdot {\bm y}$ as
\begin{equation*}
{\bm r}\cdot {\bm y}=\sum_{\ell'=1}^{\ell}\sum_{i\in V_{\ell'}}r_iy_i.
\end{equation*}
By the definition of $\mathscr P,$ if $1\leq i_1\neq i_2\leq m$ belong to the same set,
then  the equation
\begin{equation}\label{eqmodpj}
x^d_{i_1}-x^d_{i_2}=0 \pmod{p_j}
\end{equation} is one of the equations defining the curve $\mathcal{C}(d,{\bm a}\bar{b},p_j).$

Let $\gamma_j=(d, p_{j}-1).$ Since $\gamma_j \mid (p_{j}-1),$ there exists a $\gamma_j$-th primitive root of unity mod $p_j,$ say $\alpha_j.$  Then equation \eqref{eqmodpj} gives
\begin{equation}\label{solns}
x_{i_2}=\alpha^t_jx_{i_1}, \quad 0\leq t\leq \gamma_j-1.
\end{equation}
Replacing \eqref{eqmodpj} by \eqref{solns}, we can regard $\mathcal{C}(d,{\bm a}\bar{b},p_j)$ as the union of $\gamma_j$ curves. Repeating this process for all such pairs $i_1$ and $i_2,$ we see that $\mathcal{C}(d,{\bm a}\bar{b},p_j)$ is a union of $\gamma_j^{m-\ell}$ curves.  Note that we can apply \eqref{BW} provided that ${\bm r}\cdot {\bm y}$ is nonconstant along any of these curves. The exception occurs when there exists a function $\theta: V \rightarrow \{1, \alpha_j, \ldots, \alpha_j^{\gamma_j-1}\}$ such that for any $1\leq \ell' \leq \ell,$ we have
\begin{equation}\label{thetasum}
\sum_{i\in V_{\ell'}}\theta(i)r_i=0 \pmod {p_j}.
\end{equation}
In this case, we use the following trivial bound instead of \eqref{BW}:
\begin{equation}\label{exc}
\Bigg{|}\sum_{{\bm y}\in \mathcal{C}(d,{\bm a}\bar{b},p_j)}e\bigg(-\frac{b_j{\bm r}\cdot {\bm y}}{p_j}\bigg)\Bigg{|}\ll_{m, d} p_j.
\end{equation}
For fixed $\bm{a}$ and $\bm{r},$ we denote by $D(\bm{a}, \bm{r})$ the product of the prime factors $p_j$ of $q$ for which $k_j=1$ and \eqref{thetasum} holds for some $\theta.$
By \eqref{rema}, \eqref{kj2}, \eqref{pjd}, \eqref{BW}, and \eqref{exc}, we obtain
\begin{equation}\label{h}
\sum_{{\bm y}\in \mathcal{C}(d,{\bm a}\bar{b},q)}e\bigg(-\frac{{\bm r}\cdot {\bm y}}{q}\bigg)\ll_{\delta, m, d}
q^{2\delta m+1/2}D(\bm{a}, \bm{r})^{1/2}c_{m, d}^{\omega(q)},
\end{equation}
where $\omega(q)$ is the number of prime divisors of $q$, and $c_{m,d}$ is a constant depending on $m$ and $d.$
Since $q \rightarrow \infty$, we can assume $c_{m, d}^{\omega(q)}\leq q^{\delta m}.$
Hence, putting \eqref{h} into \eqref{remainder2} yields the following:
\begin{align*}
\mathcal{E}&\ll_{\delta, m, d}\frac{q^{1/2+3\delta m-m}}{N}\sum_{\substack{{\bm r} (\text{mod }q)\\{\bm r} \neq 0}}
\prod_{i=1}^m\,\min\bigg{\{} N, \frac{q}{|r_i|}\bigg{\}}\sum_{{\bm a}\in sI}^*D(\bm{a}, \bm{r})^{1/2}\\
&\ll\frac{q^{1/2+3\delta m}}{N}\sum_{\substack{{\bm r} (\text{mod }q)\\{\bm r} \neq 0}}
\prod_{i=1}^m\,\min\bigg{\{} \frac1s, \frac{1}{|r_i|}\bigg{\}}\sum_{D  \mid \tilde{q}}D^{1/2}\mathcal{N}(D),
\end{align*}
where $\mathcal{N}(D):=\#\{{\bm a}\in (sI)^*: D(\bm{a}, \bm{r})=D\}.$
If we let $\rho({\bm r}, D)$ be the proportion of integer vectors ${\bm a}$ in $(sI)^*$ such that $D(\bm{a}, \bm{r})=D,$
then
\begin{equation*}
\mathcal{N}(D)\sim \rho({\bm r}, D)s^{m-1}\vol(I).
\end{equation*}
Thus, it follows that
\begin{equation}
\mathcal{E}\ll_{\delta, m, d, I}q^{-1/2+3\delta m}s^m\sum_{D \mid \tilde{q}}D^{1/2}
\sum_{\substack{{\bm r} (\text{mod }q)\\{\bm r} \neq 0}}\rho({\bm r}, D)\prod_{i=1}^m\,\min\bigg{\{} \frac1s, \frac{1}{|r_i|}\bigg{\}}.
\end{equation}
In order to prove Lemma \ref{LemmaRemainder}, we show the following:
\begin{equation}\label{maineq}
D^{1/2}\sum_{\substack{{\bm r} (\text{mod }q)\\{\bm r} \neq 0}}\rho({\bm r}, D)\prod_{i=1}^m\,\min\bigg{\{} \frac1s, \frac{1}{|r_i|}\bigg{\}}
\ll_{\delta_0, m, d, I}q^{1/2-\delta_0}s^{-m}.
\end{equation}
Note that the number of divisors $D$ of $\tilde{q}$ is $\ll_{\epsilon}q^{\epsilon}.$
Thus, for $\delta$ small enough in terms of $m$ and $\delta_0,$ \eqref{maineq} implies
$$\mathcal{E}\ll_{\delta, \delta_0, m, d, I} q^{-\delta_0/2},$$
which completes the proof of the lemma.

Now, for a nonempty subset $\mathscr{L}$ of $V=\{1, 2, \ldots, m\},$
we let $\sum(\mathscr{L})$ denote the sum on the left-hand side of \eqref{maineq} over the vectors ${\bm r} \pmod{q}$
such that $r_i\neq 0$  if and only if $i\in \mathscr{L}.$ There are $2^m-1$ such subsums $ \sum(\mathscr{L}),$ and therefore it suffices to show, for each $\mathscr{L},$ that
\begin{equation}\label{sumL}
\sum(\mathscr{L})\ll_{m, \delta_0, d, I} D^{-1/2}q^{1/2-\delta_0}s^{-m}.
\end{equation}
For convenience, we can assume that $\mathscr{L}=\{1, 2, \ldots, L\}$ for some $1\leq L\leq m.$
It follows that
\begin{align*}
\sum(\mathscr{L})&=\sum_{0<|r_1|, \ldots, |r_L|\leq q/2}\rho({\bm r}, D)\prod_{i=1}^m\,\min\bigg{\{} \frac1s, \frac{1}{|r_i|}\bigg{\}}\\
&\leq \frac{1}{s^{m-L}}\sum_{0<|r_1|, \ldots, |r_L|\leq q/2}\rho({\bm r}, D)\prod_{i=1}^L\frac{1}{|r_i|}.
\end{align*}
We consider a $L$-tuple $\mathscr{D}=(D_1, \ldots, D_L)$ with $D_i \mid D,$ $1\leq i\leq L$,
and set $$\mathscr{M}(\mathscr{D}):=\max\{\rho(\bm r, D): 0<|r_i|\leq q/2, (r_i, D)=D_i\}.$$
Observe that
\begin{align*}
\sum_{\substack{0<|r_1|, \ldots, |r_L|\leq q/2\\(r_i, D)=D_i}}\rho({\bm r}, D)\prod_{i=1}^L\frac{1}{|r_i|}
\leq \frac{\mathscr{M}(\mathscr{D})}{D_1\cdots D_L}\sum_{0<|e_i|\leq q/(2D_i)}\prod_{i=1}^{L}\frac{1}{|e_i|}
\ll\frac{\mathscr{M}(\mathscr{D})(\log q)^L}{D_1\cdots D_L}.
\end{align*}
Note that the number of such $L$-tuples $\mathscr{D}$ is $\ll q^{\epsilon}.$ Thus,
\eqref{sumL} holds  if we can show that for each $\bm r$ with $(r_i, D)=D_i,$
\begin{equation}\label{maineq1}
\rho(\bm r, D)\ll_{m, \delta_0, d, I} D^{-1/2}q^{1/2-\delta_0-2\epsilon}s^{-L}\prod_{i=1}^{L}D_i.
\end{equation}
Obviously, \eqref{maineq1} is true if the right-hand side of \eqref{maineq1} is strictly bigger than $1.$
We therefore assume that $D_1, \ldots, D_L$ satisfy
\begin{equation*}
\prod_{i=1}^{L}D_i\leq D^{1/2}s^Lq^{-1/2+\delta_0/2+2\epsilon}.
\end{equation*}
Also, since $s\leq q^{1/(2m)-\delta_0},$ we have
\begin{equation}\label{prodDi}
\prod_{i=1}^{L}D_i\leq D^{1/2}q^{L/(2m)-1/2+(1-L)\delta_0+2\epsilon}\leq D^{1/2}q^{-\delta_0/2}.
\end{equation}
The following pages contain the proof of \eqref{maineq1}, thus completing the proof of Lemma \ref{LemmaRemainder}.

{\bf Proof of \eqref{maineq1}.}
For a prime divisor  $p_j$ of $D,$ we let
$m_j$ be the number of components $r_i=0 \pmod {p_j}.$ Then it can be easily seen that
\begin{equation}\label{prodDi2}
\prod_{i=1}^{L}D_i=\prod_{p_j \mid D}p_j^{m_j}.
\end{equation}
Consider a vector $\bm a$ which contributes to $\rho(\bm r, D)$ and the corresponding partition $\mathscr{P}(\bm a,  p_j).$
Recall that such $\bm a$ satisfies \eqref{thetasum}.  Since $r_i=0$ for $L<i\leq m,$ we can view $\mathscr{P}(\bm a,  p_j)$ as a partition of $\{1, \ldots, L\}.$ Note that if $r_i\neq 0 \pmod{p_j},$ then the subset $V_{\ell'}$ containing the index $i$
has more than one element. Thus, there are at most $\big[\frac{L-m_j}{2}\big]$ such subsets $V_{\ell'}.$
Now, in each $V_{\ell'},$ we choose the largest index $i(\ell')$ for which $r_{i(\ell')}\neq 0 \pmod{p_j}.$
Then, we see that $V_{\ell'}$ produces the following independent congruences:
$$A_i-A_{i(\ell')}=0\pmod{p_j}, ~~  i\in V_{\ell'}\setminus \{i(\ell')\}.$$
Since each of the $L-m_j$ indices $i\neq i(\ell')$ with $r_i\neq 0 \pmod{p_j}$ corresponds to exactly one of these independent congruences, the number of such congruences, say $\tilde{m}_j,$ satisfies
\begin{equation}\label{tildem}
\tilde{m}_j\geq L-m_j-\Big[\frac{L-m_j}{2}\Big].
\end{equation}
We put together all these congruences for all the prime divisors $p_j$ of $D.$
Note that $m_j=0$ for some $p_j$, since otherwise \eqref{prodDi2} implies that $D$ divides $\prod_{i=1}^{L}D_i,$
which contradicts \eqref{prodDi}. Thus, every index $i\in \{1, \ldots, L\}$ appears in the congruences.
Also, the vectors $\bm a$ which satisfy all these congruences lie on a lattice whose fundamental parallelepiped has volume
$\prod_{p_j\mid D}p_j^{\tilde{m}_j}.$

We remark that all the vectors $\bm a$ which contribute to $\rho(\bm r, D)$ are placed on several lattices.
The number of such lattices is $\ll_{\epsilon} q^{\epsilon},$ and therefore we can consider only one fixed lattice.
In the following, we count the vectors $\bm a$ in $\rho(\bm r, D)$ that satisfy a fixed set of congruences described above.

It suffices to count the $m$-tuples $(A_1, \ldots, A_m)$ which satisfy the congruences, since $\bm a$ is uniquely determined by such a $m$-tuple.
Recall that $A_m=0$ by definition. Using the condition $\bm a\in (sI)^*,$ we can bound the integers $A_{m-1}, \ldots, A_L$
by some positive constant, say $sc_I.$ Hence there are at most $(sc_I)^{m-L}$ choices for $(A_{m-1}, \ldots, A_L).$
Now fix $A_{m-1}, \ldots, A_L.$ We first consider the congruences that involve $A_{L-1}-A_L$ and put them together to obtain
$$A_{L-1}-A_L=0\pmod{d_{L-1}}$$ for some $d_{L-1}\mid D.$ This gives at most $c_I([s/d_{L-1}+1])$ possible values for $A_{L-1}.$ Similarly, we consider the congruences involving $A_{L-2}-A_{L}$  and $A_{L-2}-A_{L-1}.$ Putting them together
yields one congruence $$A_{L-2}=B_{L-2}\pmod{d_{L-2}},$$ where $d_{L-2}$ is some divisor of $D$, and $B_{L-2}\pmod{d_{L-2}}$ is an integer
uniquely determined in terms of $A_{L-1}$ and $A_{L}.$ This implies that for each value of $A_{L-1},$ there exist at most
$c_I([s/d_{L-2}+1])$ possible values for $A_{L-2}.$
We repeat this argument for $A_{L-3}, \ldots, A_1$ in order.
Here, we can see that
\begin{equation}\label{proddi}
\prod_{i=1}^{L-1}d_i=\prod_{p_j\mid D}p_j^{\tilde{m}_j}.
\end{equation}
It follows that the number of vectors $\bm a$ on the lattice is bounded by
\begin{equation*}
s^{m-L}\prod_{i=1}^{L-1}\Big(\Big[\frac{s}{d_i}\Big]+1\Big) \ll_m s^{m-L}\prod_{i=1}^{L-1}\frac{\max\{s, d_i\}}{d_i}
=s^{m-L+m_0}\prod_{\substack{i=1\\d_i<s}}^{L-1}\frac{1}{d_i},
\end{equation*}
where $m_0$ denotes the number of $d_i$'s less than $s.$
This implies that 
\begin{equation*}
\rho(\bm r, D)\ll_{\epsilon}\frac{s^{m-L+m_0}q^{\epsilon}}{\#(sI)^*\prod_{d_i<s}d_i}\ll_{I}\frac{s^{1-L+m_0}q^{\epsilon}}{\prod_{d_i<s}d_i}.
\end{equation*}
Hence \eqref{maineq1} holds provided that we show the following:
\begin{equation}\label{maineq1.5}
D\ll q^{1-3\delta_0}s^{-2m_0-2}\Big(\prod_{i=1}^LD_i^2\Big)\Big(\prod_{d_i<s}d_i^2\Big).
\end{equation}
Using $s\leq q^{1/(2m)-\delta_0},$ we see that $s^{2m}\leq q^{1-2m\delta_0}\leq q^{1-4\delta_0}.$
Therefore  \eqref{maineq1.5} follows from
\begin{equation}\label{maineq2}
D\ll q^{\delta_0}s^{2(m-1-m_0)}\Big(\prod_{i=1}^LD_i^2\Big)\Big(\prod_{d_i<s}d_i^2\Big).
\end{equation}
It remains to prove \eqref{maineq2}.
From  \eqref{prodDi2}, \eqref{proddi} and \eqref{tildem}, we can deduce that 
\begin{equation}\label{DL}
\Big(\prod_{i=1}^LD_i\Big)\Big(\prod_{i=1}^{L-1}d_i^2\Big)=\prod_{p_j\mid D}p_j^{m_j+2\tilde{m}_j}
\geq \prod_{p_j\mid D}p_j^{m_j+(L-m_j)}=D^L.
\end{equation}
Thus, we get
\begin{align*}
\Big(\prod_{i=1}^LD_i\Big)\Big(\prod_{d_i<s}d_i^2\Big)=\Big(\prod_{i=1}^LD_i\Big)\Big(\prod_{i=1}^{L-1}d_i^2\Big)
\Big(\prod_{d_i\geq s}d_i^2\Big)^{-1}\geq D^{L-2(L-1-m_0)}=D^{2m_0+2-L}.
\end{align*}
Note that if $2m_0+2> L,$ then \eqref{maineq2} is true. We thus assume $2m_0+2\leq L.$
Since $$2(m-1-m_0)\geq2m-L\geq L,$$ equation  \eqref{maineq2} holds if we show
\begin{equation}\label{maineq3}
D\ll s^L.
\end{equation}
From  \eqref{DL} and \eqref{prodDi}, we derive
\begin{equation*}
D^L\leq \Big(\prod_{i=1}^LD_i\Big)\Big(\prod_{i=1}^{L-1}d_i^2\Big) \leq D^{1/2}q^{-\delta_0/2}\Big(\prod_{i=1}^{L-1}d_i^2\Big),
\end{equation*}
which implies that
\begin{equation}\label{proddi2}
\prod_{i=1}^{L-1}d_i>D^{(L-1/2)/2}.
\end{equation}
On the other hand, we return to the aforementioned set of congruences of the form $$A_{i_1}-A_{i_2}=0 \pmod{p_j}.$$
Recall that we have $\tilde{m}_j$ such congruences for each $p_j.$ For each fixed pair $(i_1, i_2),$ we combine
all the congruences involving $A_{i_1}-A_{i_2}$ to obtain one congruence
\begin{equation}\label{di12}
A_{i_1}-A_{i_2}=0\pmod{d_{i_1, i_2}}.
\end{equation}
Note that this is just a different arrangement of the set of congruences above. In particular, we see that, by \eqref{proddi},
\begin{equation}\label{proddi3}
\prod_{1\leq i_1<i_2\leq L}d_{i_1, i_2}=\prod_{p_j\mid D}p_j^{\tilde{m}_j}=\prod_{i=1}^{L-1}d_i.
\end{equation}
It follows from \eqref{di12} that $\prod d_{i_1, i_2}$ divides
\begin{equation*}
\prod_{1\leq i_1<i_2\leq L}(A_{i_1}-A_{i_2}).
\end{equation*}
Since each $A_i$ is bounded by $sc_I,$ we have
$$\prod_{1\leq i_1<i_2\leq L}(A_{i_1}-A_{i_2})\ll_{I}s^{L(L-1)/2}.$$
Hence, by \eqref{proddi2} and \eqref{proddi3},
\begin{equation}\label{proddi4}
D^{(L-1/2)/2}<\prod_{i=1}^{L-1}d_i\ll_{I}s^{L(L-1)/2}.
\end{equation}
From \eqref{proddi4}, we deduce that $D\ll s^L,$ which proves \eqref{maineq3}.
The proof of Lemma \ref{LemmaRemainder} is therefore complete.
\end{proof}

\section{A necessary and sufficient condition}

In this section we show that the obstruction to being Poissonian along a sequence for $n^d\alpha\mod 1, d\geq 2$, is the same as for $n^2\alpha\mod 1$. This obstruction consists in the presence of large square factors in the denominators of good approximants. Theorem \ref{curious} is an easy consequence of the following result.

\begin{theorem}\label{main}
Let $d\geq 2$ be an integer, and let $\alpha$ be an irrational number for which there are infinitely many rationals $b_j/q_j$ satisfying
\begin{equation}\label{goodapprox}
\bigg{|}\alpha-\frac{b_j}{q_j}\bigg{|}<\frac{1}{q_j^{d+1}}.
\end{equation}
Then the following are equivalent.
\begin{enumerate}
\item There exists a sequence $N_j\rightarrow\infty$ with $\frac{\log N_j}{\log q_j}\rightarrow 1$ such that $n^d\alpha\mod 1$ is Poissonian along $N_j$.
\item There exists a sequence $N_j\rightarrow\infty$ with $\frac{\log N_j}{\log q_j}\rightarrow 1$ such that $n^2\alpha\mod 1$ is Poissonian along $N_j$.
\item Letting $\tilde{q}_j$ denote the square free part of $q_j$, we have \begin{equation*}
\lim_{j\to\infty}\frac{\log \tilde{q}_j}{\log q_j}=1.
\end{equation*}
\end{enumerate}
\end{theorem}

Conditions {\em (2)} and {\em (3)} are equivalent by \cite[Theorem 1]{Z03}. Hence it is enough to prove that {\em (1)} is equivalent to {\em (3)}. The implication {\em (3)} $\Longrightarrow$ {\em (1)} requires Theorem \ref{T2}. The other direction follows the arguments in \cite[Section 3]{Z03}, and requires the two following lemmas. The first is a divergency principle for $m$-correlations (see \cite[Lemma 6]{RSZ01}).
\begin{lemma}\label{div}
Let $q=uv^2$ with $v>q^{\delta}$ for some $\delta>0$. Let $\eta>1$, and suppose that $\log N/\log q>\eta$. Let $f\in C^{\infty}_c({\bf R}^{m-1})$ be a non-negative test function which is non-vanishing at the origin. Then, for every integer $b$ and every $d\geq 2$, 
\begin{equation*}
R^{(m)}\bigg(N,d,\frac{b}{q},f\bigg)\gg_{m,\delta}\frac{1}{N}f(0)\bigg(\frac{Nv}{q}\bigg)^m.
\end{equation*}
\end{lemma}
\begin{proof}
By the definition of the $m$-level correlation, 
\begin{equation*}
R^{(m)}\bigg(N,d,\frac{b}{q},f\bigg)=\frac{1}{N}\sum_{\substack{1\leq n_1,\ldots,n_m\leq N\\n_j\text{ distinct}}}f\bigg(\dots,N\bigg{\{}\frac{bn_j^d}{q}\bigg{\}}-N\bigg{\{}\frac{bn_{j+1}^d}{q}\bigg{\}},\dots\bigg).
\end{equation*}
Since $f\geq 0$, it is enough to estimate the contribution of the terms $(n_1,\dots,n_m),$ with $n_j$ distinct, such that $n_1,\dots,n_m$ are all divisible by $uv$. There are $\gg_m[N/uv]^m=[Nv/q]^m$ such $m$-tuples. If $n=uvn'$, then, since $q=uv^2$, 
\begin{equation*}
\bigg{\{} \frac{bn^d}{q}\bigg{\}}=\{bu(n')^2(un'v)^{d-2}\}=0,
\end{equation*}
and therefore
\begin{equation*}
R^{(m)}\bigg(N,d,\frac{b}{q},f\bigg)\gg_{m,\delta}\frac{1}{N}f(0)\bigg(\frac{Nv}{q}\bigg)^m,
\end{equation*}
as wanted.
\end{proof}

The next lemma allows one to pass from the $m$-level correlation of a family of finite sequences to another family, which is close enough to the original one. Let $\mathcal{N}=\{x_N(n)\colon n\leq N\}$ and $\mathcal{N}'=\{x'_N(n)\colon n\leq N\}$ be two families of sequences in $[0,1)$. We define for each $N$ the scaled distance between the corresponding sequences to be
\begin{equation*}
\varepsilon_N(\mathcal{N},\mathcal{N}'):=N\max_{n\leq N}|x_N(n)-x'_N(n)|.
\end{equation*}
Recall that the $m$-level correlation for the family $\mathcal{N}$ is defined for every $f\in C^{\infty}_c({\bf R}^{m-1})$ by
\begin{equation*}
R^{(m)}(N,\mathcal{N},f):=\frac{1}{N}\sum_{\substack{1\leq n_1,\ldots,n_m\leq N\\n_j\text{ distinct}}}F_N\Big(x_N(n_1)-x_N(n_2),\dots,x_N(n_{m-1})-x_N(n_m)\Big),
\end{equation*}
where $F_N({\bf y}):=\sum_{{\bf l}\in{\bf Z}^{m-1}}f(N({\bf l}+{\bf y}))$.

\begin{lemma}\cite[Lemma 5]{RSZ01}\label{approx}
Assume that $\mathcal{N},\mathcal{N}'\subset [0,1)$ are two families of sequences with $\varepsilon_N(\mathcal{N},\mathcal{N}')\rightarrow 0$ as $N\rightarrow \infty$. Then for every $f\in C^{\infty}_c({\bf R}^{m-1})$,
\begin{equation*}
\big{|}R^{(m)}(N,\mathcal{N},f)-R^{(m)}(N,\mathcal{N}',f)\big{|}\leq R^{(m)}(N,\mathcal{N},f_{+})\varepsilon_N(\mathcal{N},\mathcal{N}')
\end{equation*}
for $N$ sufficiently large, with $f_{+}\in C^{\infty}_c({\bf R}^{m-1}),$ a non-negative function depending only on $f$.
\end{lemma}

\begin{proof}[Proof of Theorem \ref{main}]
{\em (1)} $\Longrightarrow$ {\em (3)} We prove the contrapositive. Assume that {\em (3)} fails. Then there are infinitely many indices $j$'s and a $\delta>0$ for which in the decomposition $q_j=\tilde{q}_jv_j^2$, with $\tilde{q}_j$ square free, we have $v_j>q_j^{\delta}$. Let $N_j\rightarrow\infty$ be a sequence with $\frac{\log N_j}{\log q_j}\rightarrow 1$. For $j$ large enough we have $N_j\in [q_j^{1-\delta/2},q_j^{1+\delta/2}]$. Consider the two families of sequences \[\mathcal{N}=\Big{\{}\{\alpha n^d\}\colon 1\leq n\leq N_j\Big{\}}\text{ and } \mathcal{N}'=\bigg{\{}\Big{\{}\frac{b_jn^d}{q_j}\Big{\}}\colon 1\leq n\leq N_j\bigg{\}}.\]
Let $f\in C^{\infty}_c({\bf R}^{m-1})$ be a non-negative function that does not vanish at the origin. We want to argue that the $m$-level correlation $R^{(m)}(N_j, \mathcal{N},f)$ diverges as $j\rightarrow\infty$ for $m$ large enough. Since $f$ is non-negative, we can restrict ourselves to considering the contribution of the $m$-tuples ${\bf x}=(x_1,\dots,x_m)\in J^m$, where $J=\Big{\{}1,\dots,\floor*{q_j^{1-\delta/2}}\Big{\}}\subseteq\{1,\dots,N_j\}$. On such ${\bf x}$, the scaled distance between $\mathcal{N}$ and $\mathcal{N}'$ is
\[\varepsilon_{N_j}(\mathcal{N},\mathcal{N}')=N_j\max_{n\in J}\bigg{|}n^d\alpha-\frac{n^db_j}{q_j}\bigg{|}\leq q_j^{1+\delta/2}q_j^{d-\delta}\frac{1}{q_j^{d+1}}=q_j^{-\delta/2}.\]
By Lemma \ref{approx}, we can thus pass to the family $\mathcal{N}'$. It is enough to prove that the contribution of the $m$-tuples ${\bf x}\in J^m$ to $R^{(m)}(N_j, \mathcal{N}',f)$ makes it diverge for $m$ large enough. By the definition of $\mathcal{N}'$ and Lemma \ref{div}, 
\[R^{(m)}(N_j,\mathcal{N}',f)=R^{(m)}\bigg(N_j,d,\frac{b_j}{q_j},f\bigg)\gg_{m,\delta}\frac{1}{N_j}f(0)\bigg(\frac{N_jv_j}{q_j}\bigg)^m\geq q_j^{(\delta/2)(m-1)-1},\]
and therefore $R^{(m)}(N_j, \mathcal{N},f)$ diverges as $j\rightarrow\infty$ while keeping $m$ and $\delta$ fixed, provided $m>1+\frac{2}{\delta}$.

{\em (3)} $\Longrightarrow$ {\em (1)} We construct the required sequence $N_j$. For every integer $k\geq 2$ we define an integer $j_k$ in the following way. Apply Theorem \ref{T2} for every $m\in\{2,\dots,k\}$ with $\delta_0=(8k)^{-1}$. There exists $\delta=\delta(k)>0$ such that for $q\to\infty$ satisfying $\tilde{q}\geq q^{1-\delta(k)}$ and for every residue $b\mod q$ with $(b,q)=1$, the $m$-level correlation for the sequence $n^db\,(\text{mod } q), 1\leq n\leq N$, where $N=\floor{q^{1-1/(4k)}}$, is Poissonian for every $m\in\{2,\dots,k\}$. We apply this to every pair $(b_j,q_j)$. By {\em (3)} there exists $j_k$ such that, for every $j\geq j_k$, we have $\tilde{q}_j\geq q_j^{1-\delta(k)}$ and the $m$-level correlation for the sequence $\{n^db_j\,(\text{mod } q_j)\}, 1\leq n\leq \floor{q_j^{1-1/(4k)}}$, is Poissonian for $m\in\{2,\dots,k\}$. It follows from \eqref{goodapprox} and Lemma \ref{approx} that for $j\geq j_k$ the $m$-level correlation for $m\in\{2,\dots, k\}$ for the sequence $\{n^d\alpha\}, 1\leq n\leq \floor{q_j^{1-1/(4k)}}$, is Poissonian. With $j_k$ defined as above, we now put $N_j=\floor{q_j^{1-1/(4k)}}$ for all those $j\geq j_k$ such that $j<j_{k+1}$. The sequence $N_j$ thus defined has the required properties.
\end{proof}

	\section{Acknowledgements}
	
	The authors are grateful to Bruce Berndt for useful comments and suggestions.

\bibliographystyle{alpha}

\begin{thebibliography}{NRSW89}
					
					\bibitem[AL94]{AL94} W. W. Adams and P. Loustaunau, An introduction to Gr\"obner bases. Graduate Studies in Mathematics, 3. {\em American Mathematical Society, Providence, RI,} 1994.
					\bibitem[BT77]{BT77} M. V. Berry and M. Tabor, Level clustering in the regular spectrum, {\em Proc. Royal Soc. London A} {\bf 356} (1977), 375--394.
					\bibitem[BZ00]{BZ00} F. P. Boca and A. Zaharescu, Pair correlation of values of rational functions (mod $p$). {\em Duke Math. J.} {\bf 105} (2000), no. 2, 267--307.
					\bibitem[B66]{B66} E. Bombieri, On exponential sums in finite fields. {\em Amer. J. Math.} {\bf 88} (1966), 71--105.
					\bibitem[CGI87]{CGI87} G. Casati, I. Guarneri, and F. M. Izrailev, Statistical properties of the quasi-energy spectrum of a simple integrable system, {\em Phys. Lett. A} {\bf 124} (1987), 263--266.
					\bibitem[D51]{D51} H. Davenport, On a principle of Lipschitz. {\em J. London Math. Soc.} {\bf 26} (1951), 179--183. Corrigendum: On a principle of Lipschitz. {\em J. London Math. Soc.} {\bf 26} (1964), 580.
					\bibitem[DZ19]{DZ19} A. Dunn and A. Zaharescu, The twisted second moment of modular half integral weight $L$-functions. Preprint. \href{https://arxiv.org/abs/1903.03416}{https://arxiv.org/abs/1903.03416}.
								\bibitem[GTZ88]{GTZ88} P. Gianni, B. Trager, and G. Zacharias,
Gr\"obner bases and primary decomposition of polynomial ideals.
Computational aspects of commutative algebra. {\em J. Symbolic Comput.} {\bf 6} (1988), no. 2-3, 149--167.



\bibitem[K89]{K89} G. Karpilovsky, Topics in field theory. North-Holland Mathematics Studies, 155. Notas de Matem\'atica [Mathematical Notes], 124. {\em North-Holland Publishing Co.,} Amsterdam, 1989.

\bibitem[KR99]{KR99} P. Kurlberg, and Z. Rudnick, The distribution of spacings between quadratic residues. {\em Duke Math. J.} {\bf 100} (1999), no. 2, 211--242.
\bibitem[PBG89]{PBG89} A. Pandey, O. Bohigas, and M. J. Giannoni, Level repulsion in the spectrum of two-dimensional harmonic oscillators, {\em J. Phys. A} {\bf 22} (1989), 4083--4088.
\bibitem[RS98]{RS98} Z. Rudnick and P. Sarnak, The pair correlation function of fractional parts of polynomials. {\em Comm. Math. Phys.} {\bf 194} (1998), no. 1, 61--70.			
\bibitem[RSZ01]{RSZ01} Z. Rudnick, P. Sarnak, and A. Zaharescu, The distribution of spacings between the fractional parts of $n^2\alpha$. {\em Invent. Math.} {\bf 145} (2001), no. 1, 37--57.
\bibitem[Sc76]{Sc76} W. M. Schmidt, Equations over finite fields. An elementary approach. Lecture Notes in Mathematics, Vol. 536. {\em Springer-Verlag, Berlin-New York,} 1976.
\bibitem[Sh12]{Sh12} I. E. Shparlinski, Modular hyperbolas. {\em Jpn. J. Math.} {\bf 7} (2012), no. 2, 235--294.
\bibitem[S\'o58]{S58} V. S\'os, On the distribution mod 1 of the sequence $n\alpha$, {\em Ann. Univ. Sci. Budapest. E\"otv\"os Sect. Math.} {\bf 1} (1958), 127--134.
\bibitem[St98]{St98} M. Stillman, Gr\"obner Bases: a Tutorial. Available at  \href{https://www3.risc.jku.at/research/theorema/Groebner-Bases-Bibliography/gbbib_files/publication_190.pdf}{www3.risc.jku.at}.
\bibitem[Sw59]{Sw59} S. Swierczkowski, On succesive settings of an arc on the circumference of a circle, {\em Fund. Math.} {\bf 46} (1959), 187--189.
								\bibitem[W48]{W48} A. Weil, Sur les Courbes Alg\'ebriques et les Vari\'et\'es qui s'en D\'eduisent, Hermann. Paris 1948.
	\bibitem[Z03]{Z03} A. Zaharescu, Correlation of fractional parts of $n^2\alpha$. {\em Forum Math.} {\bf 15} (2003), no. 1, 1--21.
	
	
	
 							\end{thebibliography}

\end{document}